\documentclass[12pt]{amsart}
\usepackage{amssymb}
\usepackage[all]{xy}
\usepackage{tabularx}
\usepackage{booktabs}
\usepackage{caption}

\newtheorem{thm}{Theorem}[section]
\newtheorem{lem}[thm]{Lemma}

\newtheorem{prop}[thm]{Proposition}
\newtheorem{ex}[thm]{Example}

\newtheorem*{prob*}{Open problem}

\theoremstyle{definition}

\newtheorem{defi}[thm]{Definition}

\theoremstyle{remark}

\newtheorem{rem}[thm]{Remark}
\newtheorem*{rem*}{Remark}

\DeclareMathOperator{\Aut}{Aut}
\DeclareMathOperator{\Lie}{Lie}

\newcommand{\kringel}{\mathbin{\raise1pt\hbox{$\scriptstyle\circ$}}}
\newcommand{\pkt}{\mathbin{\raise0pt\hbox{$\scriptstyle\bullet$}}}

\newcommand{\C}{\mathbb{C}}

\newcommand{\K}{\mathbb{K}}

\newcommand{\R}{\mathbb{R}}

\newcommand{\ad}{\mathop{\rm ad}}

\newcommand{\Der}{\mathop{\rm Der}}

\newcommand{\La}{\mathfrak{a}}
\newcommand{\Lb}{\mathfrak{b}}

\newcommand{\Lf}{\mathfrak{f}}
\newcommand{\Lg}{\mathfrak{g}}
\newcommand{\Lh}{\mathfrak{h}}

\newcommand{\Ll}{\mathfrak{l}}

\newcommand{\Ln}{\mathfrak{n}}

\newcommand{\Lr}{\mathfrak{r}}
\newcommand{\Ls}{\mathfrak{s}}

\newcommand{\CC}{\mathcal{C}}

\newcommand{\im}{\mathop{\rm im}}

\newcommand{\al}{\alpha}
\newcommand{\be}{\beta}

\newcommand{\la}{\lambda}

\newcommand{\ra}{\rightarrow}

\renewcommand{\phi}{\varphi}

\begin{document}

\title[Abelian subalgebras]{Abelian ideals of maximal 
dimension for solvable Lie algebras}

\author[D. Burde]{Dietrich Burde}
\author[M. Ceballos]{Manuel Ceballos}
\address{Fakult\"at f\"ur Mathematik\\
Universit\"at Wien\\
Nordbergstr. 15\\
1090 Wien \\
Austria}
\email{dietrich.burde@univie.ac.at}
\address{Departamento Geometria y Topologia\\
Universidad de Sevilla\\ 
Sevilla\\
Spain}
\email{mceballos@us.es}

\date{\today}

\thanks{The first author was supported by the Austrian Science Fund FWF, project P21683}
\thanks{The second author was supported by MTM 2010-19336 and FEDER}

\begin{abstract}
We compare the maximal dimension of abelian subalgebras and the maximal
dimension of abelian ideals for finite-dimensional Lie algebras. We show that
these dimensions coincide for solvable Lie algebras over an algebraically closed field
of characteristic zero. We compute this invariant for all complex nilpotent Lie algebras 
of dimension $n\le 7$. 
Furthermore we study the case where there exists an abelian subalgebra of
codimension $2$. Here we explicitly construct an abelian ideal of codimension $2$  
in case of nilpotent Lie algebras. 
\end{abstract}

\maketitle

\section{Introduction}

Let $\Lg$ be a finite-dimensional Lie algebra. Denote by $\al(\Lg)$ 
the maximal dimension of an abelian subalgebra of $\Lg$, and by $\be(\Lg)$
the maximal dimension of an abelian ideal of $\Lg$. Both invariants
are important for many subjects. First of all they are very useful invariants
in the study of Lie algebra contractions and degenerations. There is a large
literature, in particular for low-dimensional Lie algebras, see 
\cite{GRH,BU10,NPO,SEE,GOR}, and the references given therein. \\
Secondly, there are several results concerning the question of 
how big or small these maximal dimensions can be, compared to the dimension of the 
Lie algebra. For references see \cite{STE,RIL,MIL}. The results show, roughly
speaking, that a Lie algebra of large dimension contains abelian subalgebras of
large dimension. For example, the dimension of a nilpotent Lie algebra $\Lg$ satisfying
$\al(\Lg)=\ell$ is bounded by $\dim (\Lg)\le \frac{\ell(\ell+1)}{2}$ \cite{STE,RIL}.
There is a better bound for $2$-step nilpotent Lie algebras, see \cite{MIL2}.
If $\Lg$ is a complex solvable Lie algebra with $\al(\Lg)=\ell$, then we have
$\dim (\Lg)\le \frac{\ell(\ell+3)}{2}$, see \cite{MIL}. In general,
$\dim (\Lg)\le \frac{\ell(\ell+17)}{2}$ for any complex Lie algebra $\Lg$ with  $\al(\Lg)=\ell$,
see \cite{MIL}. \\
For semisimple Lie algebras $\Ls$ the invariant $\al(\Ls)$ has been completely 
determined by Malcev \cite{MAL}. Since there are no nontrivial abelian ideals in $\Ls$, we have
$\be(\Ls)=0$. Recently the study of abelian ideals in a Borel subalgebra
$\Lb$ of a simple complex Lie algebra $\Ls$ has drawn considerable attention.
We have indeed $\al(\Ls)=\be(\Lb)$, and this number can be computed purely in terms
of certain root system invariants, see \cite{SUT}. The result is reproduced for the
interested reader in table $1$.
\begin{table}[htp]
\centering
\caption{The invariant $\al$ for simple Lie algebras}
\newcolumntype{K}{>{\centering \arraybackslash $}X <{$}}
\begin{tabularx}{.5\textwidth}{*{3}{K}}
\toprule
\Ls & \dim (\Ls) & \alpha(\Ls) \\
\midrule
A_n,\,n\ge 1 & n(n+2) & \lfloor (\frac{n+1}{2})^2 \rfloor \\
\midrule
B_3 & 21  & 5  \\ 
\midrule
B_n,\, n\ge 4  & n(2n+1) & \frac{n(n-1)}{2}+1 \\ 
\midrule
C_n,\,n\ge 2 & n(2n+1) & \frac{n(n+1)}{2} \\ 
\midrule
D_n,\,n\ge 4 & n(2n-1) & \frac{n(n-1)}{2} \\ 
\midrule
G_2 & 14 & 3 \\ 
\midrule
F_4 & 52 & 9 \\ 
\midrule
E_6 & 78 & 16 \\ 
\midrule
E_7 & 133 & 27 \\ 
\midrule
E_8 & 248 & 36 \\ 
\bottomrule
\end{tabularx}
\end{table}
Furthermore Kostant found a relation of these invariants to discrete series representations 
of the corresponding Lie group, and to powers of the Euler product \cite{KO1,KO2}.
In fact, there are many more results concerning the invariants $\al$ and $\be$ for 
simple Lie algebras and their Borel subalgebras. \\
In this paper, we want to point out the following interesting result: if 
$\Lg$ is a solvable Lie algebra over an algebraically closed field of characteristic zero, 
then we have $\al(\Lg)=\be(\Lg)$. This means, given an abelian subalgebra of maximal 
dimension $m$ there exists also an abelian ideal of dimension $m$. \\
For a given value of $\al(\Lg)$ the dimension of $\Lg$ is bounded in terms of this
value, as we have mentioned before. Hence it is natural to ask what we can say
on $n$-dimensional Lie algebras $\Lg$ where the value of $\al(\Lg)$ is close to $n$.
Indeed, if $\al(\Lg)=n$, then $\Lg$ is abelian, $\al(\Lg)=\be(\Lg)$, and we are done. 
If $\al(\Lg)=n-1$, then also $\be(\Lg)=n-1$. This means that $\Lg$ has an abelian
ideal of codimension $1$ and is almost abelian. In particular, $\Lg$ is $2$-step solvable.
In this case the structure of $\Lg$, and even all its degenerations are quite 
well understood, see \cite{GOR}. \\
After these two easy cases it is reasonable to consider Lie algebras
$\Lg$ satisfying $\al(\Lg)=n-2$. Here we can classify all such non-solvable complex Lie algebras.
However, for solvable Lie algebras we cannot expect to obtain a classification, not even
in the nilpotent case. In fact, there exist even characteristically nilpotent Lie algebras
$\Lg$ with  $\al(\Lg)=n-2$. On the other hand we know that 
$\al(\Lg)=\be(\Lg)$, so that there is an abelian ideal of codimension $2$. 
For many problems concerning the cohomology of nilpotent Lie algebras
the subclass of those having an abelian ideal of codimension $1$ or $2$ is very
important, see \cite{ARM,POT} and the references given therein. \\
The structure of the paper is as follows. In section $2$ we prove basic 
results concerning the invariants $\al(\Lg)$ and $\be(\Lg)$. For a 
Levi decomposition $\Lg=\Ls\ltimes \Lr$ of $\Lg$ we show that 
$\al(\Ls\ltimes \Lr)\le \al(\Ls)+\al(\Lr)$. The main result of this section is, as
mentioned above, that $\al(\Lg)=\be(\Lg)$ for solvable Lie algebras over an algebraically
closed field of characteristic zero. An example is given that the statement is not true
in general for the field of real numbers. \\
In section $3$ we construct an abelian ideal of codimension $1$ for a Lie algebra
$\Lg$ satisfying $\al(\Lg)=n-1$. In section $4$ we show that Lie algebras $\Lg$ with
$\al(\Lg)=n-2$ are solvable or isomorphic to $\Ls\Ll_2(\C)\oplus \C^{\ell}$ for some $\ell\ge 0$.
In section $5$ we study nilpotent Lie algebras $\Lg$ with $\al(\Lg)=n-2$ and 
explicitly construct an abelian ideal of codimension $2$. Finally the invariants are 
computed for all complex nilpotent Lie algebras of dimension $n\le 7$. This is accompanied 
by a remark on Lie algebra degenerations, where these invariants are really useful.

\section{The invariants $\al(\Lg)$ and $\be (\Lg)$}

\begin{defi}
Let $\Lg$ be a Lie algebra of dimension $n$ over a field $K$. If not stated otherwise
we assume that $K$ is the field of complex numbers. 
Some results will also hold for other fields, but we are mainly interested in the case
of complex numbers.
Consider the following invariants of $\Lg$:
\begin{align*}
\al(\Lg) & = \max \{\dim (\La) \mid \La \text{ is an abelian subalgebra of }\Lg\},\\
\be(\Lg) & = \max \{\dim (\Lb) \mid \Lb \text{ is an abelian ideal of }\Lg\}.
\end{align*}
\end{defi}

An abelian subalgebra of maximal dimension is maximal abelian with respect to inclusion. 
However, a maximal abelian subalgebra need not be of maximal dimension:

\begin{ex}
Let $\Lf_n$ be the standard graded filiform nilpotent Lie algebra of dimension
$n$. Let $(e_1,\ldots ,e_n)$ be a standard basis, such that $[e_1,e_i]=e_{i+1}$ 
for $2\le i\le n-1$. Then $\La=\langle e_1,e_n\rangle$
is a maximal abelian subalgebra of dimension $2$, but $\al(\Lf_n)=\be(\Lf_n)=n-1$.
\end{ex}

Clearly we have $\be(\Lg)\le \al (\Lg)$. In general, the two invariants are
different. A complex semisimple Lie algebra $\Ls$ has no abelian ideals, hence
$\be(\Ls)=0$. We already saw in table $1$ that this is not true for the invariant 
$\al(\Ls)$. As mentioned before the following result holds, see \cite{SUT}:

\begin{prop}
Let $\Ls$ be a complex simple Lie algebra and $\Lb$ be a Borel subalgebra
of $\Ls$. Then the maximal dimension of an abelian ideal in $\Lb$ coincides
with the maximal dimension of a commutative subalgebra of $\Ls$, i.e.,
$\al(\Ls)=\be(\Lb)$. Furthermore the number of abelian ideals in $\Lb$ is
$2^{{\rm rank}(\Ls)}$.
\end{prop}

This implies $\al(\Lb)=\be(\Lb)$, because we have
$\al(\Lb)\le \al(\Ls)=\be(\Lb)$, since $\al$ is monotone:

\begin{lem}\label{2.4}
The invariant $\al$ is monotone and additive: for a subalgebra $\Lh\le\Lg$ of $\Lg$ we have
$\al(\Lh)\le \al(\Lg)$, and for two Lie algebras $\La$ and $\Lb$
we have $\al(\La\oplus\Lb) =\al(\La)+\al(\Lb)$.
\end{lem}

The invariant $\be$ need not be monotone. For example, consider a Cartan subalgebra 
$\Lh$ in $\Lg=\Ls\Ll_2(\C)$. Then $\be(\Lh)=1>0=\be(\Lg)$. We also have the following
result:

\begin{lem}\label{levi}
Let $\Lg$ be a complex Lie algebra with Levi decomposition $\Lg=\Ls\ltimes \Lr$.
Then $\al(\Ls\ltimes \Lr)\le \al(\Ls)+\al(\Lr)$.
\end{lem}

\begin{proof}
Let $\La$ be an abelian subalgebra in $\Lg$ of maximal dimension and 
$\pi\colon \Ls\ltimes \Lr \ra \Ls$, $(x,a)\ra (x,0)$ be the projection. Restricting
this homomorphism of Lie algebras to $\La$ yields
\[
\dim(\La)=\dim (\ker (\pi_{\La}))+\dim (\im (\pi_{\La})).
\]
Since $\im (\pi_{\La})$ is the homomorphic image of a subalgebra of $\Ls\ltimes \Lr$,
we may assume that  $\im (\pi_{\La})$ is an abelian subalgebra of $\Ls$. In particular we
have $\dim (\im (\pi_{\La}))\le \al(\Ls)$. Furthermore we have $\ker (\pi_{\La})=\La\cap \Lr$.
Hence  $\ker (\pi_{\La})$ is an abelian subalgebra of $\Lr$ and we have 
$\dim (\ker (\pi_{\La}))\le \al(\Lr)$. Finally we obtain
\begin{align*}
\al(\Ls\ltimes \Lr) & =\dim (\La) \\
 & = \dim (\ker (\pi_{\La}))+\dim (\im (\pi_{\La})) \\
 & \le \al(\Ls)+\al(\Lr).
\end{align*}
\end{proof}

The fact that $\al(\Lb)=\be(\Lb)$ for a Borel subalgebra $\Lb$ of a complex simple
Lie algebra can be generalized to all complex solvable Lie algebras. 

\begin{prop}\label{2.5}
Let $\Lg$ be a solvable Lie algebra over an algebraically closed field $K$ of
characteristic zero. Then $\be(\Lg)=\al(\Lg)$.
\end{prop}

\begin{proof}
The result follows easily from the proof of Theorem $4.1$ of \cite{ELO}. 
For the convenience of the reader we give the details.
Let $G$ be the adjoint algebraic group of $\Lg$. This is the smallest
algebraic subgroup of $\Aut(\Lg)$ such that its Lie algebra $\Lie (G)$ contains
$\ad (\Lg)$. Then $\Lie (G)$ is the algebraic hull of $\ad (\Lg)$. Since
$\ad (\Lg)$ is solvable, so is  $\Lie (G)$. Therefore $G$ is a connected solvable
algebraic group. Let $m=\al(\Lg)$. Consider the set $\CC$ of all commutative
subalgebras of $\Lg$ of dimension $m$. This is, by assumtion, a non-empty set,
which can be considered as a subset of the Grassmannian $Gr(\Lg,m)$, which is
an irreducible complete algebraic variety. Hence $\CC$ is a non-empty complete
variety, and $G$ operates morphically on it, mapping each commutative subalgebra
$\Lh$ on $g(\Lh)$, for $g\in G$. By Borel's fixed point theorem, $G$ has a fixed 
point $I$ in $\CC$, i.e., a subalgebra $I$ of $\Lg$ with $g(I)=I$ for all
$g\in G$. In particular we have $\ad(x)(I)\subseteq I$ for all $x\in \Lg$. 
Hence $I$ is an abelian ideal of dimension $m$ of $\Lg$.
\end{proof}

Borel's fixed point theorem relies on the closed orbit lemma. As a corollary
one can also obtain the theorem of Lie-Kolchin. We note that the assumptions on
$K$ are really necessary. The next example shows that we need the field to be
algebraically closed.

\begin{ex}
Let $\Lg$ be the solvable Lie algebra of dimension $4$ over $\R$ defined by
\begin{align*}
[x_1,x_2]& =x_2-x_3,\quad [x_1,x_4]=2x_4, \\
[x_1,x_3]& =x_2+x_3,\quad [x_2,x_3]=x_4 \\
\end{align*}
Then, over $\R$, we have $\al(\Lg)=2$, but $\be(\Lg)=1$.
\end{ex}

Let $\K$ be equal to $\R$ or $\C$.
Obviously, $\langle x_3,x_4\rangle$ is an abelian subalgebra of dimension $2$ over
$\K$. Assume that $\al(\Lg)=3$. Then $\Lg$ is almost abelian, hence
$2$-step solvable. This is impossible, as $\Lg$ is $3$-step solvable. Hence $\al(\Lg)=2$
over $\K$. \\
Assume that $I$ is a $2$-dimensional abelian ideal over $\K$. It is easy to see that we
can represent $I$ as $\langle \al x_2+\be x_3,x_4\rangle$ with $\al,\be \in \K$.
Obviously neither $x_2$ nor $x_3$ can belong to $I$. Hence $\al\neq 0$ and $\be \neq 0$. 
We have $\al x_2+\be x_3\in I$ and $[x_1,\al x_2+\be x_3]=(\al+\be)x_2+(\be-\al)x_3\in I$.
This implies $(\al^2+\be^2)x_3\in I$, hence $\al^2+\be^2=0$.
This is a contradiction over $\R$, so that $\be(\Lg)=1$ in this case.
Over $\C$ we may take $\al=1$ and $\be=i$, and  $I=\langle x_2+ix_3,x_4\rangle$ 
is a  $2$-dimensional abelian ideal. \\[0.2cm]

The next two lemma's are well known. We state them just for further reference.

\begin{lem}\label{mub}
Let $\Lg$ be a complex, non-abelian, nilpotent Lie algebra of dimension $n$. Then 
\[
\frac{\sqrt{8n+1}-1}{2}\le \al(\Lg)\le n-1
\]
\end{lem}

\begin{proof}
The estimate is given in \cite{ELO} for $\be(\Lg)$. It also holds for
$\al(\Lg)$ since $\al(\Lg)\ge \be(\Lg)$. 
\end{proof}

\begin{lem}\label{2.8}
The center $Z(\Lg)$ of $\Lg$ is contained in any 
abelian subalgebra of maximal dimension.
\end{lem}

\begin{proof}
An abelian subalgebra $\La$ of maximal dimension is self-centralizing, 
i.e., $\La=Z_{\Lg}(\La)=\{x\in \Lg\mid [x,\La]=0 \}$. 
Since $Z(\Lg)\subset Z_{\Lg}(\La)$, the claim follows.
\end{proof}

\section{Abelian subalgebras of codimension $1$}

Let $\Lg$ be a Lie algebra satisfying $\al(\Lg)=n-1$. Such a Lie algebra is
$2$-step solvable, and their structure is well known (see \cite{GOR}, section $3$).
We will show that $\be(\Lg)=n-1$ without using proposition $\ref{2.5}$. 
Our proof will be constructive. We do not only show the existence of an abelian ideal 
of dimension $n-1$, but really construct such an ideal from a given abelian subalgebra of dimension
$n-1$. Note that Lie algebras $\Lg$ with $\be(\Lg)=n-1$ are called
{\it almost abelian}.

\begin{prop}
Let $\Lg$ be a $n$-dimensional Lie algebra over a field of characteristic zero 
satisfying $\al(\Lg)=n-1$. Then we have $\be(\Lg)=n-1$ and an ideal of 
codimension $1$ can be constructed explicitly.
\end{prop}

\begin{proof}
Let $\La$ be an abelian subalgebra of dimension $n-1$. If $[\Lg,\Lg]\subseteq \La$,
then $\La$ is also an abelian ideal, and we are done. Otherwise we choose a basis
$(e_1,\ldots ,e_n)$ for $\Lg$ such that $\La=\langle e_2,\ldots ,e_n \rangle$.
We have $[e_j,e_{\ell}]=0$ for all $j,\ell \ge 2$.
There exists a $k\ge 2$ such that $[e_1,e_k]$ is not contained in $\La$.
We may assume that $k=2$ by relabelling $e_2$ and $e_k$. 
For $j\ge 2$ let
\[
[e_1,e_j]=\al_{j1}e_1+\al_{j2}e_2+\cdots + \al_{jn}e_n.
\]
We have $\al_{21}\neq 0$. Rescaling $e_1$ we may assume that $\al_{21}=1$.
Using the Jacobi identity we have for all $j\ge 2$
\begin{align*}
0 & = [e_1,[e_2,e_j]] \\
  & = -[e_2,[e_j,e_1]]-[e_j,[e_1,e_2]] \\
  & = -\al_{j1}[e_1,e_2]+[e_1,e_j]
\end{align*}
This implies $[e_1,e_j]=\al_{j1}[e_1,e_2]$ and $[e_1,\al_{j1}e_2-e_j]=0$ for all $j\ge 2$. 
Let $v_j=\al_{j1}e_2-e_j$. Note that all $v_j$ lie in the center of $\Lg$, and that 
the derived subalgebra $[\Lg,\Lg]$ is $1$-dimensional, generated by $[e_1,e_2]$.
Now define
\[
I:=\langle [e_1,e_2], v_3,\ldots ,v_n\rangle.
\]
This is an abelian subalgebra of dimension $n-1$ which contains the derived
subalgebra $[\Lg,\Lg]$. Hence $I$ is an abelian ideal of maximal dimension
$n-1$, and we have $\be(\Lg)=n-1$.
\end{proof}

\section{Abelian subalgebras of codimension $2$}

Let $\Lg$ be a complex Lie algebra of dimension $n$ satisfying $\al(\Lg)=n-2$. 
We will show that $\Lg$ must be solvable except for the cases
$\Ls\Ll_2(\C)\oplus \C^{\ell}$, for $\ell=n-3\ge 0$. We use the convention that
the Lie algebra $\Ls\Ll_2(\C)$ is included in this family, for $\ell=0$.

\begin{prop}
Let $\Lg$ be a complex Lie algebra with $\dim (\Lg)=n$ and $\al(\Lg)=n-2$. Then either $\Lg$ is 
isomorphic to one of the Lie algebras $\Ls\Ll_2(\C)\oplus \C^{\ell}$, or $\Lg$ is
a solvable Lie algebra. 
\end{prop}

\begin{proof}
Let $\Lg=\Ls\ltimes \Lr$ be a Levi decomposition, where $\Lr$ denotes the
solvable radical of $\Lg$. For a semisimple Levi subalgebra $\Ls$ we have
\[
\al(\Ls)\le \dim (\Ls)-2,
\] 
where equality holds if and only if $\Ls$ is $\Ls\Ll_2(\C)$. This follows from
table $1$ and lemma $2.4$. By lemma $\ref{levi}$ we have
$\al(\Ls\ltimes \Lr)\le \al(\Ls)+\al(\Lr)$. 
Assume that $\Ls\neq 0$. Then it follows that
\begin{align*}
\al(\Lg) & \le \al(\Ls)+\al(\Lr) \\
         & \le \dim(\Ls)-2+\dim(\Lr) \\
         & = n-2.
\end{align*}
Since we must have equality, it follows that $\Ls$ is isomorphic to $\Ls\Ll_2(\C)$,
and $\al(\Lr)=\dim(\Lr)$. Therefore $\Lr$ is abelian and 
$\Lg\simeq \Ls\Ll_2(\C)\ltimes_{\phi}\C^{\ell}$ with a homomorphism
$\phi\colon \Ls\Ll_2(\C) \ra \Der (\C^{\ell})$. This Lie algebra contains an abelian
subalgebra of codimension $2$ if and only if $\phi$ is trivial. Indeed, the Lie bracket
is given by $[(x,a),(y,b)]=([x,y], \phi(x)b-\phi(y)a)$, for $x,y\in \Ls\Ll_2(\C)$ and
$a,b\in \C^{\ell}$. Since there is an abelian subalgebra of codimension $2$,
there must be a nonzero element $(x,0)$ commuting with all elements $(0,b)$, i.e.,
$(0,0)=[(x,0),(0,b)]=(0,\phi(x)b)$ for all $b\in \C^{\ell}$.
It follows that $\ker(\phi)$ is non-trivial. Since  $\Ls\Ll_2(\C)$ is simple, $\phi=0$. \\
In the other remaining case we have $\Ls=0$. In that case, $\Lg$ is solvable.
\end{proof}

It is easy to classify such Lie algebras in low dimensions.

\begin{prop}
Let $\Lg$ be a complex Lie algebra of dimension $n$ and $\al(\Lg)=n-2$.
\begin{itemize}
\item[(1)] For $n=3$ it follows $\Lg\simeq \Ls\Ll_2(\C)$.
\item[(2)] For $n=4$, $\Lg$ is isomorphic to one of the following Lie algebras: 

\vspace*{0.5cm}
\begin{center}
\begin{tabular}{l|l}
$\Lg$ & Lie brackets \\[0.1cm]
\hline
 & \\[-0.35cm]
$\Lg_1=\Lr_2(\C) \oplus \Lr_2(\C)$ & $[e_1,e_2]=e_2, \,[e_3,e_4]=e_4$ \\[0.1cm]
$\Lg_2=\Ls\Ll_2(\C) \oplus \C$ & $[e_1,e_2]=e_2, \,[e_1,e_3]=-e_3,\,[e_2,e_3]=e_1$\\[0.1cm]
$\Lg_3$ & $[e_1,e_2]=e_2, \,[e_1,e_3]=e_3,\,[e_1,e_4]=2e_4,\,[e_2,e_3]=e_4$\\[0.1cm]
$\Lg_4(\al),\,\al\in \C$ & $[e_1,e_2]=e_2, \,[e_1,e_3]=e_2+\al e_3,\,
[e_1,e_4]=(\al +1)e_4,\,[e_2,e_3]=e_4$
\end{tabular}
\end{center}
\end{itemize}
\end{prop}
\vspace*{0.3cm}
\begin{proof}
The proof is straightforward, using a classification of low-dimensional
Lie algebras, e.g., the one given in \cite{BU18}.
Note that $\Lg_4(\al)\simeq \Lg_4(\be)$ if and only if $\al\be=1$ or $\al=\be$.
\end{proof}

\section{Nilpotent Lie algebras}

In a nilpotent Lie algebra $\Lg$ any subalgebra of codimension $1$ is automatically
an ideal. Hence given an abelian subalgebra of maximal dimension $n-1$ we obtain
an abelian ideal of dimension $n-1$. In particular, $\al(\Lg)=n-1$ for a nilpotent
Lie algebra implies $\be(\Lg)=\al(\Lg)$, and we can explicitly provide such ideals. 
We are able to extend this result to the case $\al(\Lg)=n-2$. Given an abelian
subalgebra of dimension $n-2$ we can construct an abelian ideal of dimension $n-2$.
This is non-trivial, since the abelian subalgebra of maximal dimension $n-2$ 
need not be an ideal in general.
Of course, the existence of such an ideal follows already from proposition $\ref{2.5}$,
as does the equality $\al(\Lg)=\be(\Lg)$. However, the existence proof is not constructive.
Our proof will be constructive and elementary, which might be more
appropriate to our special situation. 

\begin{prop}\label{3.1}
Let $\Lg$ be a nilpotent Lie algebra of dimension $n$ over a field of
characteristic zero satisfying $\al(\Lg)=n-2$.
Then there exists an algorithm to construct an abelian ideal of dimension $n-2$
from an abelian subalgebra of dimension $n-2$. In particular we have
$\be(\Lg)=\al(\Lg)$.
\end{prop}

\begin{proof}
Let $\La$ be an abelian subalgebra of $\Lg$ of maximal dimension $n-2$. Choose a basis
$(e_3,\ldots ,e_n)$ for $\La$. The normalizer of $\La$,
\[
N_{\Lg}(\La)=\{x \in \Lg\mid [x,\La]\subseteq \La \}
\]
is a subalgebra strictly containing $\La$. We may assume that $N_{\Lg}(\La)$
has dimension $n-1$, because otherwise $N_{\Lg}(\La)=\Lg$, implying that $\La$ is
already an abelian ideal of maximal dimension $n-2$.  \\
We may extend the basis of $\La$ to a basis $(e_1,\ldots ,e_n)$ of $\Lg$, such that
$N_{\Lg}(\La)=\langle e_2,\ldots ,e_n\rangle$. Because $N_{\Lg}(\La)$ has codimension $1$,
it is an ideal in $\Lg$. In particular we have
\[
[e_1,N_{\Lg}(\La)]\subseteq N_{\Lg}(\La).
\]
On the other hand, $[e_1,\La]$ is not contained in $\La$, since $e_1$ is not in
$N_{\Lg}(\La)$. Hence there exists a vector $e_k$ such that $[e_1,e_k]$ is not in
$\La$. By relabelling $e_3$ and $e_k$ we may assume that $k=3$. Hence
writing
\[
[e_1,e_j]=\al_{j2}e_2+\cdots +\al_{jn}e_n
\]
for $j\ge 2$, we may assume that $\al_{32}=1$, i.e., $[e_1,e_3]=e_2+\al_{33}e_3+\cdots 
+\al_{3n}e_n$.

\begin{lem}
The following holds:
\begin{itemize}
\item[(1)] We have $[e_2,e_j]=\al_{j2}[e_2,e_3]$ for all $j\ge 3$.
\item[(2)] The element $[e_2,e_3]$ is nonzero and contained in the center of $\Lg$.
\item[(3)] The normalizer $N_{\Lg}(\La)$ is two-step nilpotent.
\item[(4)] We have $[N_{\Lg}(\La), v_j]=0$ for all $j\ge 3$, where $v_j=\al_{j2}e_3-e_j$.
\end{itemize}
\end{lem}

\begin{proof}
The first statement follows from the Jacobi identity. We have, for all $j\ge 3$,
\begin{align*}
0 & = [e_1,[e_3,e_j]] \\
  & = -[e_3,[e_j,e_1]]-[e_j,[e_1,e_3]] \\
  & = -\al_{j2}[e_2,e_3]+[e_2,e_j].
\end{align*}
Concerning $(2)$, assume first that $[e_2,e_3]=0$. Then the subalgebra given by 
$\langle e_2,e_3,v_4,\ldots ,v_n\rangle $
would be an abelian subalgebra of dimension $n-1$, with the $v_j$ defined as in $(4)$.
This is a contradiction to $\al(\Lg)=n-2$. Hence $[e_2,e_3]$ is non-zero. 
Since $e_2 \in N_{\mathfrak{g}}(\mathfrak{a})$, we have that $[e_2,e_3] \in \La$.
We write
\[
[e_2,e_3]=\be_{33}e_3+\cdots +\be_{3n}e_n.
\]
We have $[e_3,[e_2,e_3]]=0$. Now we have
\[
[e_2,[e_2,e_3]]=(\be_{33}\al_{32}+\cdots +\be_{3n}\al_{n2})[e_2,e_3].
\]
Since $\ad(e_2)$ is nilpotent, it follows $[e_2,[e_2,e_3]]=0$. 
In the same way, $[e_1,[e_2,e_3]]=[e_2,[e_1,e_3]]-[e_3,[e_1,e_2]]=\la [e_2,e_3]$,
so that $[e_1,[e_2,e_3]]=0$, because $\ad(e_1)$ is nilpotent.
Finally, $[e_j,[e_2,e_3]]=0$ for all $j\ge 3$, since $[e_2,e_3]\in \La$. It follows that
$[e_2,e_3]$ lies in the center of $\Lg$. \\
To show $(3)$, note that $[N_{\Lg}(\La),N_{\Lg}(\La)]$ is generated by $[e_2,e_3]$, so that
\[
[N_{\Lg}(\La),N_{\Lg}(\La)]\subseteq Z(\Lg).
\]
This proves $(3)$. \\
The statement $(4)$ follows from $(1)$.
\end{proof}

Now let $\La_1=\langle v_4,\ldots v_n \rangle$. This is an abelian subalgebra
$\La_1\subseteq \La\subseteq \Lg$ of dimension $n-3$. There exists an integer
$\ell \ge 1$ satisfying 
\begin{align*}
\ad(e_1)^{\ell-1}(e_2) & \not\in \La_1,\\
\ad(e_1)^{\ell}(e_2) & \in \La_1,\\
\end{align*}
because $\ad(e_1)$ is nilpotent. We define
\[
I:=\langle \ad(e_1)^{\ell-1}(e_2), v_4,\ldots ,v_n \rangle
\]
We will show that $I$ is an abelian ideal of maximal dimension $n-2$. 
First of all, $I$ is a subalgebra of dimension $n-2$. It is also abelian:
because $N_{\Lg}(\La)$ is an ideal, $\ad(e_1)^k(e_2)\in N_{\Lg}(\La)$ for all
$k\ge 0$. Then
\begin{align*}
[\ad(e_1)^k(e_2),v_j] & =[\la_2e_2+\cdots \la_ne_n, \al_{j2}e_3-e_j]\\
  & = \la_2\al_{j2}[e_2,e_3]-\la_2[e_2,e_j] \\
  & = 0. 
\end{align*}
It remains to show that $I$ is an ideal, i.e., that $\ad(e_i)(I)\subseteq I$ for
all $i\ge 1$.
We have
\begin{align*}
[e_1,\ad(e_1)^{\ell-1}(e_2)]& =\ad(e_1)^{\ell}(e_2)\in \La_1\subseteq I,\\
[e_k,\ad(e_1)^{\ell-1}(e_2)]& \in [N_{\Lg}(\La),N_{\Lg}(\La)]\subseteq Z(\Lg)\subseteq I, 
\end{align*}
for all $k\ge 2$. Here we have used lemma $\ref{2.8}$ to conclude that $Z(\Lg)\subseteq I$.
Also, $[e_k,v_j]=0 \in I$ for all $k\ge 2$ and $j\ge 4$. It remains to show that
\[
[e_1,v_j]\in I \text{ for all } j\ge 4.
\]
We have
\begin{align*}
[e_2,[e_1,v_j]] & = [e_1,[e_2,v_j]]+[v_j,[e_1,e_2]] \\
 & = 0.
\end{align*}
This implies that $[e_1,v_j]$ commutes with all elements from $I$. If it were not
in $I$, then $\langle [e_1,v_j], I\rangle$ would be an abelian subalgebra of dimension
$n-1$, which is impossible. It follows that $[e_1,v_j]\in I$.
\end{proof}

\begin{rem}
There is also an algorithm to compute $\al(\Lg)$ for an arbitrary complex Lie algebra
of finite dimension, see \cite{CNT}.
\end{rem}

In connection with the toral rank conjecture (TRC), which asserts that 
\[
\dim  H^{\ast}(\Lg,\C)\ge 2^{\dim Z(\Lg)}
\]
for any finite-dimensional, complex nilpotent Lie algebra, there are interesting 
examples of nilpotent Lie algebras $\Lg$ given, with $\be(\Lg)=n-2$, of 
dimension $n\ge 10$, see \cite{POT}. These algebras also have the property that
all its derivations are singular. An obvious question here is whether there exist
characteristically nilpotent Lie algebras (CNLAs) $\Lg$ of dimension $n$ with
$\al(\Lg)=n-2$. This is indeed the case.

\begin{ex}
The Lie algebra of dimension $n=7$ defined by $[x_1,x_i]=x_{i+1}, 2\le i\le 6$ and
$[x_2,x_3]=x_6+x_7,\; [x_2,x_4]=x_7$ is characteristically nilpotent, i.e., all of its
derivations are nilpotent, see \cite{MAG}. Furthermore it satisfies $\al(\Lg)=n-2=5$.
\end{ex}

We can find such examples in all dimensions $n\ge 7$. This suggests that nilpotent Lie algebras 
$\Lg$ with $\al(\Lg)=n-2$ are not so easy to understand. The algebra in this example is 
{\it filiform nilpotent}, i.e., has maximal nilpotency class with respect to its dimension.
In this case we can say something more on $\al(\Lg)$. 

\begin{defi}
Let $\Lg$ be a nilpotent Lie algebra, and $C^1(\Lg)=\Lg$, $C^i(\Lg)=[\Lg,C^{i-1}(\Lg)]$.
Then $\Lg$ is called {\it $k$-abelian}, if $k$ is the smallest positive integer
such that the ideal $C^k(\Lg)$ is abelian. 
\end{defi}

For a nilpotent $k$-abelian Lie algebra $\Lg$ we have $\dim (C^k(\Lg))\le \be (\Lg)$. 
In general equality does not hold. However, if $\Lg$ is filiform nilpotent of dimension
$n\ge 6$ with $k\ge 3$, then we do have equality:

\begin{prop}
Let $\Lg$ be a $k$-abelian filiform Lie algebra of dimension $n\ge k+3\ge 6$. 
Then $\be(\Lg)=\al(\Lg)=\dim (C^k(\Lg))=n-k$, and $C^k(\Lg)$ is the unique abelian ideal 
of maximal dimension. We have
\[
\left\lceil \frac{n}{2} \right\rceil \le \be(\Lg) \le n-3.
\]
\end{prop}

\begin{proof}
We may choose an adapted basis $(e_1,\ldots ,e_n)$ for $\Lg$, see \cite{BU13}.
Then $[e_1,e_i]=e_{i+1}$ for all $2\le i\le n-1$, and $C^j(\Lg)=\langle
e_{j+1},\ldots ,e_n \rangle$ with $\dim (C^j(\Lg))=n-j$
for all $j\ge 2$. By assumption $C^k(\Lg)$ is abelian, but $C^{k-1}(\Lg)$ is not.
We claim that every abelian ideal $I$ in $\Lg$ is contained in $C^k(\Lg)$.
This will finish the proof. Suppose that there is an abelian ideal $I$ 
which is not contained in $C^k(\Lg)$. 
We will show that this implies $C^{k-1}(\Lg)\subseteq I$, so that $I$ cannot be abelian, 
a contradiction.  
Let $x=\al_1e_1+\cdots +\al_ne_n$ be a nontrivial element of $I$ not lying in $C^k(\Lg)$, i.e.,
with $\al_i\neq 0$ for some $i<k+1$. 
If $\al_1\neq 0$, then for all $2\le j\le n-1$ we have
$[e_j,x]=-\al_1e_{j+1}+\al_2[e_j,e_2]+\cdots +\al_n[e_j,e_n]\in I$.
It follows that $\langle e_3, \cdots , e_n \rangle =C^2(\Lg)\subseteq I$. Since $k>2$ this 
implies that $I$ is not abelian, a contradiction. 
Let $1<i<k+1$ be minimal such that $\al_i\neq 0$. Then for all $0\le j\le n-i$
we have $\ad(e_1)^j(x)=\al_ie_{i+j}+\cdots +\al_{n-j}e_n\in I$.
It follows that $\langle e_i,\ldots ,e_n\rangle \subseteq I$. Indeed, for $j=n-i$ we
have $\al_ie_n\in I$, then $\al_ie_{n-1}\in I$, and so on until $\al_ie_i\in I$.
Since $i\le k$ we have  $C^{k-1}(\Lg)\subseteq I$ and we are finished. Finally, we have
$3\le k\le \lfloor \frac{n}{2}\rfloor $, so that we obtain the estimate on $\be(\Lg)$.
\end{proof}

\begin{rem}
If $\Lg$ is filiform with $k=2$, then $\Lg$ is $2$-step solvable and we have 
$\be(\Lg)=n-1$ or  $\be(\Lg)=n-2$. Indeed,
if $\Lg$ is the standard graded filiform $\Lf_n$ of dimension $n\ge 3$,
then $\be(\Lg)=n-1$ and $I=\langle e_2,\ldots ,e_n\rangle$ is an abelian ideal
of dimension $n-1$. Otherwise $C^2(\Lg)=[\Lg,\Lg]$ is an abelian ideal of maximal 
dimension, so that $\be(\Lg)=n-2$. 
\end{rem}

The invariant $\al(\Lg)$ for complex nilpotent Lie algebras has been determined
up to dimension $6$ in connection with degenerations \cite{BU10},\cite{SEE}, and
for real Lie algebras of dimension $6$ in \cite{MAG}, appendix $4.4$.

We want to give a list here, thereby correcting a few typos in \cite{SEE}.
In dimension $7$ there is no list for $\al(\Lg)$, as far as we know.
We use the classification of nilpotent Lie algebras up to dimension $7$ by 
Magnin \cite{MAG}, and for dimension $6$ also by de Graaf \cite{DGR} and Seeley \cite{SEE}, 
to give tables for $\al(\Lg)$. The result for the indecomposable algebras in dimension 
$n\le 5$ is as follows:
\vspace*{0.5cm}
\begin{center}
\begin{tabular}{c|c|c|c}
 $\Lg$ & $\dim(\Lg)$ & Lie brackets & $\al(\Lg)$ \\
\hline
$\Ln_3$     & $3$ & $[e_1,e_2]=e_3$ & $2$ \\
$\Ln_4$     & $4$ & $[e_1,e_2]=e_3, \,[e_1,e_3]=e_4$ & $3$ \\
$\Lg_{5,6}$ & $5$ & $[e_1,e_2]=e_3,\, [e_1,e_3]=e_4,\, [e_1,e_4]=e_5,\, [e_2,e_3]=e_5$ & $3$ \\
$\Lg_{5,5}$ & $5$ & $[e_1,e_2]=e_3,\, [e_1,e_3]=e_4,\, [e_1,e_4]=e_5$ & $4$  \\
$\Lg_{5,3}$ & $5$ & $[e_1,e_2]=e_4,\, [e_1,e_4]=e_5,\, [e_2,e_3]=e_5$ & $3$  \\
$\Lg_{5,4}$ & $5$ & $[e_1,e_2]=e_3,\, [e_1,e_3]=e_4,\, [e_2,e_3]=e_5$ & $3$  \\
$\Lg_{5,2}$ & $5$ & $[e_1,e_2]=e_4,\, [e_1,e_3]=e_5$ & $4$ \\
$\Lg_{5,1}$ & $5$ & $[e_1,e_3]=e_5,\, [e_2,e_4]=e_5$ & $3$  \\
\end{tabular}
\end{center}
\vspace*{0.5cm}
For $n=6$ we have:
\begin{center}
\begin{tabular}{c|c|c|c}
Magnin & de Graaf & Seeley & $\al(\Lg)$ \\
\hline
$\Lg_{6,20}$ & $L_{6,14}$ & $12346_E$ & $3$ \\
$\Lg_{6,18}$ & $L_{6,16}$ & $12346_C$ & $3$ \\
$\Lg_{6,19}$ & $L_{6,15}$ & $12346_D$ & $4$ \\
$\Lg_{6,17}$ & $L_{6,17}$ & $12346_B$ & $4$ \\
$\Lg_{6,15}$ & $L_{6,21}(1)$ & $1346_C$ & $4$ \\
$\Lg_{6,13}$ & $L_{6,13}$ & $1246$ & $4$ \\
$\Lg_{6,16}$ & $L_{6,18}$ & $12346_A$ & $5$ \\
$\Lg_{6,14}$ & $L_{6,21}(0)$ &  $2346$ & $4$ \\
$\Lg_{6,9}$ & $L_{6,19}(1)$ & $136_A$ & $4$ \\
$\Lg_{6,12}$ & $L_{6,11}$ & $1346_B$ & $4$ \\
$\Lg_{5,6}\oplus\C$ & $L_{6,6}$ & $1+1235_B$ & $4$ \\
%
$\Lg_{6,5}$ &  $L_{6,24}(1)$ & $246_E$ & $4$ \\
$\Lg_{6,10}$ &  $L_{6,20}$ & $136_B$ & $4$ \\
$\Lg_{6,11}$ &  $L_{6,12}$ & $1346_A$ & $4$ \\
$\Lg_{5,5}\oplus\C$ &  $L_{6,7}$ & $1+1235_A$ & $5$ \\
$\Lg_{6,8}$ &  $L_{6,24}(0)$ & $246_D$ & $4$ \\
$\Lg_{6,4}$ &  $L_{6,19}(0)$ & $246_B$ & $4$ \\
$\Lg_{6,7}$ &  $L_{6,23}$ & $246_C$ & $4$ \\
$\Lg_{6,2}$ &  $L_{6,10}$ & $146$ & $4$ \\
$\Lg_{6,6}$ &  $L_{6,25}$ & $246_A$ & $5$ \\
$\Lg_{5,4}\oplus \C$ &  $L_{6,9}$ & $1+235$ & $4$ \\
$\Lg_{5,3}\oplus \C$ &  $L_{6,5}$ & $1+135$ & $4$ \\
$\Ln_3\oplus \Ln_3$ &  $L_{6,22}(1)$ & $13+13$  & $4$ \\
$\Ln_4\oplus \C^2$ &  $L_{6,3}$ & $2+124$  & $5$ \\
$\Lg_{6,1}$ &  $L_{6,22}(0)$ & $26$ & $4$ \\
$\Lg_{6,3}$ &  $L_{6,26}$ & $36$ & $4$ \\
$\Lg_{5,2}\oplus \C$ & $L_{6,8}$ & $1+25$  & $5$ \\
$\Lg_{5,1}\oplus \C$ & $L_{6,4}$ & $1+15$  & $4$ \\
$\Ln_3\oplus \C^3$ &  $L_{6,2}$ & $3+13$  & $5$ \\
$\C^6$ &  $L_{6,1}$ & $0$  & $6$ \\
\end{tabular}
\end{center}
\vspace*{0.5cm}
The Hasse diagram for the degenerations of nilpotent Lie algebras in dimension $6$ is given
in the end of the paper. For more details on degenerations see \cite{BU18}.
There are some typos in \cite{SEE} which we found
by computing all degenerations again. This is not part of this paper, but will be
published elsewhere. The Hasse diagram is only listed for the interested reader as an appendix, 
to demonstrate that the computation of $\al(\Lg)$ has interesting applications. 
The diagram also gives a good control for the computation of $\al(\Lg)$, since it is well known
that $\al(\Lg)\le \al(\Lh)$ if $\Lg \ra_{\rm deg} \Lh$, see \cite{BU18}. \\
In dimension $7$ we use the classification of Magnin to compute $\al(\Lg)$ for all indecomposable,
complex nilpotent Lie algebras of dimension $7$, using an algorithm of \cite{CNT}.
Note that $4\le \al(\Lg)\le 6$ in this case, see lemma $\ref{mub}$. The result is as follows:\\[0.2cm]
\begin{align*}
\al(\Lg) = 4 : \quad \Lg  & = \Lg_{7,0.1},\,\Lg_{7,0.4(\la)},\,  \Lg_{7,0.5},\,
\Lg_{7,0.6},\, \Lg_{7,0.7},\, \Lg_{7,0.8},\, \Lg_{7,1.02},\, \Lg_{7,1.03},\, \Lg_{7,1.1(i_{\la}), 
\la\neq 1},\, \Lg_{7,1.1(ii)},\\
 & \hspace{0.56cm} \Lg_{7,1.1(iii)},\, \Lg_{7,1.1(iv)},\, \Lg_{7,1.1(v)},\, \Lg_{7,1.1(vi)},\, 
\Lg_{7,1.2(i_{\la}),\la\neq 1},\, \Lg_{7,1.2(ii)},\, \Lg_{7,1.2(iii)},\, \Lg_{7,1.2(iv)},\\
 & \hspace{0.56cm} \Lg_{7,1.3(i_{\la}),\la\neq 0},\, \Lg_{7,1.3(ii)},\, \Lg_{7,1.3(iii)},\, \Lg_{7,1.3(iv)},\,
 \Lg_{7,1.3(v)},\, \Lg_{7,1.5},\, \Lg_{7,1.8},\, \Lg_{7,1.11},\, \Lg_{7,1.14},\\
 & \hspace{0.56cm} \Lg_{7,1.17},\, \Lg_{7,1.19},\, \Lg_{7,1.20},\, \Lg_{7,1.21},\,
 \Lg_{7,2.1(i_{\la}),\la\neq 0,1},\, \Lg_{7,2.1(ii)},\, \Lg_{7,2.1(iii)},\, \Lg_{7,2.1(iv)},\, \Lg_{7,2.1(v)},\\
 & \hspace{0.56cm} \Lg_{7,2.2},\, \Lg_{7,2.4},\, \Lg_{7,2.5},\, \Lg_{7,2.6},\, \Lg_{7,2.10},\,
 \Lg_{7,2.12},\, \Lg_{7,2.13},\, \Lg_{7,2.17},\, \Lg_{7,2.23},\, \Lg_{7,2.26},\, \Lg_{7,2.28},\\
 & \hspace{0.56cm} \Lg_{7,2.29},\, \Lg_{7,2.30},\, \Lg_{7,2.34},\, \Lg_{7,2.35},\, \Lg_{7,2.37},\,
  \Lg_{7,3.1(i_{\la}),\la\neq 0,1},\, \Lg_{7,3.1(iii)},\, \Lg_{7,3.13},\, \Lg_{7,3.18},\\
 & \hspace{0.56cm}  \Lg_{7,3.22},\, \Lg_{7,4.4}.
\end{align*}

\begin{align*}
\al(\Lg) = 5 : \quad \Lg  & = \Lg_{7,0.2},\,\Lg_{7,0.3},\,  \Lg_{7,1.01(i)},\,
\Lg_{7,1.01(ii)},\, \Lg_{7,1.1(i_{\la}),\la=1},\, \Lg_{7,1.2(i_{\la}),\la=1},\, \Lg_{7,1.3(i_{\la}),\la=0},
\, \Lg_{7,1.4},\\
 & \hspace{0.56cm} \Lg_{7,1.6},\, \Lg_{7,1.7},\, \Lg_{7,1.9},\, \Lg_{7,1.10},\, 
\Lg_{7,1.12},\, \Lg_{7,1.13},\, \Lg_{7,1.15},\, \Lg_{7,1.16},\, \Lg_{7,1.18},\, \Lg_{7,2.1(i_{\la}),
\la=0,1}, \\
 & \hspace{0.56cm} \Lg_{7,2.7},\, \Lg_{7,2.8},\, \Lg_{7,2.9},\, \Lg_{7,2.11},\,
 \Lg_{7,2.14},\, \Lg_{7,2.15},\, \Lg_{7,2.16},\, \Lg_{7,2.18},\, \Lg_{7,2.19},\, \Lg_{7,2.20},\, \Lg_{7,2.21},\\
 & \hspace{0.56cm} \Lg_{7,2.22},\, \Lg_{7,2.24},\, \Lg_{7,2.25},\, \Lg_{7,2.27},\, \Lg_{7,2.31},\, 
\Lg_{7,2.32},\, \Lg_{7,2.33},\, \Lg_{7,2.36},\, \Lg_{7,2.38},\, \Lg_{7,2.39}, \\ 
 & \hspace{0.56cm} \Lg_{7,2.40},\,  \Lg_{7,2.41},\, \Lg_{7,2.42},\, \Lg_{7,2.43},\, \Lg_{7,2.44},\, 
\Lg_{7,2.45},\, \Lg_{7,3.1(i_{\la}), \la=0,1},\,\Lg_{7,3.3},\,\Lg_{7,3.4},\,\Lg_{7,3.5},  \\
 & \hspace{0.56cm} \Lg_{7,3.6},\, \Lg_{7,3.7},\, \Lg_{7,3.8},\, \Lg_{7,3.9},\, \Lg_{7,3.10},\,
  \Lg_{7,3.11},\, \Lg_{7,3.12},\, \Lg_{7,3.14},\, \Lg_{7,3.15},\, \Lg_{7,3.16},\, \Lg_{7,3.17},\\
 & \hspace{0.56cm}  \Lg_{7,3.19},\,\Lg_{7,3.21},\, \Lg_{7,3.23}\, \Lg_{7,3.24},\, \Lg_{7,4.1},\, 
\Lg_{7,4.3}.\\[0.2cm] 
\al(\Lg) = 6 : \quad \Lg  & = \Lg_{7,2.3},\,\Lg_{7,3.2},\,  \Lg_{7,3.20},\,\Lg_{7,4.2}.
\end{align*}

\newpage

$$
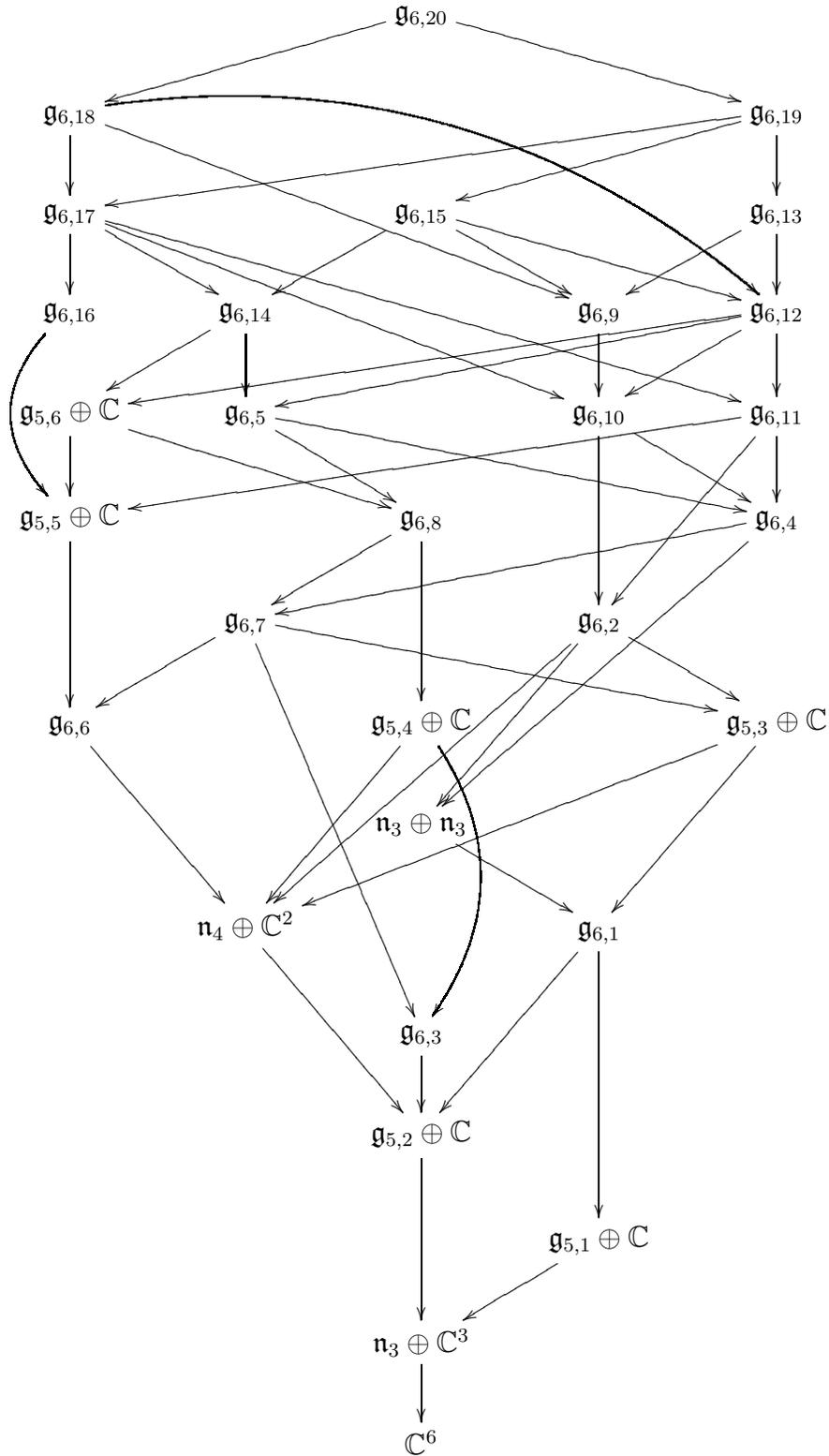
\begin{xy}
\xymatrix{
                         &            & \Lg_{6,20} \ar[lld] \ar[rrd]  &    &                  \\
\Lg_{6,18}\ar[d]\ar[ddrrr] \ar @/^3pc/[rrrrdd] &   &    &    &  \Lg_{6,19}\ar[d] \ar[dllll]\ar[dll]\\
\Lg_{6,17}\ar[d]\ar[dr]\ar[ddrrr]\ar[ddrrrr]& & \Lg_{6,15}\ar[dl]\ar[dr]\ar[drr] &
& \Lg_{6,13}\ar[d]\ar[dl] \\
\Lg_{6,16}  \ar @/_2pc/[dd]  & \Lg_{6,14}\ar[dl]\ar[d] & & \Lg_{6,9}\ar[d]
& \Lg_{6,12}\ar[dllll]\ar[dlll]\ar[dl]\ar[d] \\
\Lg_{5,6}\oplus\C\ar[d]\ar[drr] & \Lg_{6,5} \ar[dr]\ar[drrr] &  & \Lg_{6,10}\ar[dd]\ar[dr]
& \Lg_{6,11}\ar[dllll]\ar[ddl]\ar[d] \\
\Lg_{5,5}\oplus\C\ar[dd] &            & \Lg_{6,8}\ar[dl]\ar[dd]  & & \Lg_{6,4}\ar[dlll]\ar[dddll] \\
                         & \Lg_{6,7}\ar[dl]\ar[ddddr]\ar[drrr]  &
& \Lg_{6,2}\ar[ddl] \ar[dddll] \ar[dr] &         \\
\Lg_{6,6}\ar[ddr]  &     & \Lg_{5,4}\oplus\C  \ar[ddl] \ar @/^2pc/[ddd] &
& \Lg_{5,3}\oplus\C \ar[ddlll] \ar[ddl]\\
                         &            &  \Ln_3\oplus \Ln_3 \ar[dr] &      &               \\
                         & \Ln_4\oplus \C^2 \ar[ddr] &    &  \Lg_{6,1}\ar[ddl]\ar[ddd]  &         \\
                         &            &  \Lg_{6,3}\ar[d]    &          &                 \\
                         &            & \Lg_{5,2}\oplus\C \ar[dd] &     &                 \\
                         &            &      & \Lg_{5,1}\oplus\C \ar[dl] &           \\
                         &            & \Ln_3\oplus\C^3\ar[d]  &             &              \\
                         &            & \C^6              &                  &
 }
\end{xy}
$$

\newpage

\section{Acknowledgment}

We would like to thank Wolfgang Moens for helpful discusssions.

\end{document}